\documentclass[11pt]{amsart}
\usepackage{amsfonts}
\usepackage{graphicx}
\usepackage{amscd}
\usepackage{amsmath}
\usepackage{amssymb}
\usepackage{amsfonts}
\usepackage{setspace}
\usepackage[hmargin=1.2in, vmargin=1.5in]{geometry}

\newtheorem{theorem}{Theorem}[section]
\newtheorem{lemma}[theorem]{Lemma}
\newtheorem{proposition}[theorem]{Proposition}
\theoremstyle{definition}

\newtheorem{example}[theorem]{Example}
\newtheorem{corollary}[theorem]{Corollary}
\newtheorem{remark}[theorem]{Remark}

\numberwithin{equation}{section}

\begin{document}

\title[Weyl asymptotics for tensor products of operators]{Weyl asymptotics for tensor products of operators and Dirichlet divisors}
\author[T. Gramchev]{Todor Gramchev}
\address{Dipartimento di Matematica e Informatica, Universit\`a di Cagliari, Via Ospedale 72, 09124 Cagliari, Italy}
\email{todor@unica.it}

\author[S. Pilipovic]{Stevan Pilipovi\'{c}}
\address{Institute of Mathematics, University of Novi Sad, trg. D. Obradovica 4, 21000 Novi Sad, Serbia}
\email{stevan.pilipovic@uns.dmi.ac.rs}
%
\author[L. Rodino]{Luigi Rodino}
\address{Dipartimento di Matematica, Universit\`a di Torino, Via Carlo Alberto 10, 10123 Torino, Italy}
\email{luigi.rodino@unito.it}
\author[J. Vindas]{Jasson Vindas}
\address{Department of Mathematics, Ghent University, Krijgslaan 281 Gebouw S22, B-9000 Gent, Belgium}
\email{jvindas@cage.Ugent.be}




\begin{abstract}
We study the counting function of the
eigenvalues for tensor products of operators,
and their perturbations,
in the context of
Shubin classes and closed manifolds. We emphasize connections
with problems of analytic number theory, concerning in particular generalized Dirichlet divisor functions. 
\end{abstract}

\maketitle

\section{Introduction}

As well known, there are deep connections between spectral theory and analytic number theory. One main topic is given by Weyl formula for self-adjoint partial differential operators  or pseudo-differential operators. Namely, the leading term in the expansion of the counting function $N(\lambda)$ of the the eigenvalues $\leq \lambda$  is recognized to be proportional to the volume of the region defined by the $\lambda$-level surfaces of the symbol, and in turn to the number of the lattice points belonging to the region. Even, for relevant classes of operators, each point of the lattice corresponds exactly to one of eigenvalues, counted according to the multiplicity, and the computation of $N(\lambda)$ leads in a natural way to problems of number theory. Let us refer for example to \cite{BrPer,bruning}, \cite{gotze}--\cite{kaplitski}, \cite{rodino}--\cite{shubin}. In this order of ideas, the attention will be fixed here on operators of the form of tensor products
\begin{equation}\label{prod}
P=P_1\otimes...\otimes P_p
\end{equation}
where the operators $P_j$, $j=1,\ldots, p$, are self-adjoint, say strictly positive (pseudo-differential) operators on corresponding Hilbert spaces with eigenvalues $\{\lambda^{(j)}_k\}_{k=1}^\infty, j=1,...,p.$ Then, the eigenvalues of $P$ are products of the form $\lambda_{k_1}^{(1)}...\lambda_{k_p}^{(p)}$ and the eigenfunctions are tensor products of the corresponding eigenfunctions, cf. \cite{BrPer,shechter} for the general functional analytic setting. Hence, 
\begin{equation}\label{prod2}
N_P(\lambda)= \#\left\{(k_1,\dots,k_p)\in\mathbb{N}^{p}:\ \lambda^{(1)}_{k_1}\lambda^{(2)}_{k_2}\dots\lambda^{(p)}_{k_{p}}\leq\lambda\right\}.
\end{equation}
The computation of $N_p(\lambda)$ meets then some classical divisor counting problems. To give a simple example, consider the Hermite operators 
\begin{equation}\label{prod3}
H_j=\frac{1}{2}(-\partial^2_{x_j}+x_j^2)+\frac{1}{2}, \ \ \ j=1,2.
\end{equation}
Writing for short $H_1$ and $H_2$
for $H_1\otimes I_2$ and $I_1\otimes H_2,$ we define the tensorized Hermite operator $H=H_1\otimes H_2.$ In applications, $H$ is sometimes used as a substitute for the standard two dimensional Hermite operator
$H_1+H_2,$ producing the same eigenfunctions, i.e., two dimensional Hermite functions. The distribution of eigenfunctions, counted with multiplicity, is however quite different, being related to the distribution of the prime numbers.  In fact, the eigenvalues of the one-dimensional  Hermite operator, 
normalized as above, are the positive integers; therefore, (\ref{prod2}) reads in this case as
\begin{equation}\label{prod4}
N_H(\lambda)=D(\lambda)=\sum_{n=1}^{[\lambda]} d(n), \ \ \ \lambda \geq 1,
\end{equation}
where $d(n)$ denotes the number of divisors of $n$  and $[\lambda]$ stands for the integral part of $\lambda$. Dirichlet proved  in 1849 that
\begin{equation}\label{prod5}
D(\lambda)=\lambda\log \lambda
+(2\tilde{\gamma}-1)\lambda+E(\lambda),
\end{equation}
where $\tilde{\gamma}$ is the Euler-Mascheroni constant and $E(\lambda)=O(\lambda^{1/2}).$
The first term on the right hand side of (\ref{prod5}) can be easily recognized as the volume of the hyperbolic region defined by the symbol of $H$, whereas
the optimal growth order of the rest $E(\lambda)$ is a long-standing open problem in the analytic theory of numbers, see for example \cite{apostol,hardy,huxley,ivic,SteSha}.

Natural generalizations of the Hermite operators $H_j$ in (\ref{prod3}) are the global pseudo-differen\-tial operators of M. Shubin \cite{helfer,BBR,NR,shubin}. If $P_j$ is globally elliptic self-adjoint in these classes, then the Weyl formula yields
\begin{equation}\label{prod6}
N_{P_j}(\lambda)\sim A_j\lambda^{\alpha_j},
\end{equation}
where $\alpha_j=2n_j/m_j$, with $m_j$ the order of $P_j$ and 
$n_j$ the space dimension. The constant $A_j$
depends on the symbol
 of $P_j$, according to the Weyl formula. Note that the tensorized product in (\ref{prod}) is not any longer globally elliptic on $\mathbb R^n, n=n_1+...+n_p.$ 

The first aim of the present paper will be to deduce from (\ref{prod6}) an asymptotic expansion for the spectral counting function $N_P(\lambda)$ in (\ref{prod2}). 
Particular attention will be devoted to lower order terms of the asymptotic expansion for some particular cases. As an example, define $H_j$ as in (\ref{prod3}) and consider now
 \begin{equation}\label{prod7}
 H^{\vec\beta}=H_1^{\beta_1}H_2^{\beta_2},
 \end{equation}
 where $\vec{\beta}=(\beta_1,\beta_2)$ is a couple of positive integers with $\beta_2\neq\beta_1.$ Then, we shall prove that
 \begin{equation}\label{prod8}
N_{H^{\vec\beta}}(\lambda)=\zeta(\beta_2/\beta_1)\lambda^{1/\beta_1}+\zeta(\beta_1/\beta_2)\lambda^{1/\beta_2}+O(\lambda^{1/(\beta_1+\beta_2)}),
 \end{equation}
where $\zeta(z)$ is the Riemann zeta function, analytically continued in the complex plane for $z\neq 1.$
 
 As we shall also detail in the paper, parallel results can be obtained when $P_j$ in (\ref{prod}) are elliptic self-adjoint pseudo-differential operators on a smooth compact manifold. In this case (\ref{prod6}) is valid with $\alpha_j=n_j/m_j$, see \cite{hormander}.
 
 Let us finally describe what, to the best of our knowledge, was already known about tensor products of pseudo-differential operators and their spectrum,
 as well as what is new in our paper.  
An algebra of ``bisingular" pseudo-differential operators on the product of two manifolds $M_1\times M_2,$  containing $P_1\otimes P_2$ with $P_1$ and $P_2$ beign classical pseudo-differential operators on $M_1, M_2$, respectively, was studied by Rodino \cite{rodino}
in connection with the multiplicative property of the Atiyah-Singer index \cite{AS}. The spectral properties of this class were recently studied by Battisti \cite{batisti}. The variant for the Shubin type operators has been considered in \cite{BaGrPiRo}. These results give a general framework for the study of the example (\ref{prod4}) with the expansion (\ref{prod5}), and provide as well the leading term in the expansion (\ref{prod8}) for the example (\ref{prod7}). Let us also mention the articles \cite{gprw1,gprw2}, where starting from the twisted Laplacian of M. W. Wong \cite{wong}, similar problems of Dirichlet divisor-type were met. The operators in \cite{gprw1,gprw2,wong} are not tensor products, but they can be reduced to the form (\ref{prod}) by conjugating with a Fourier integral operator, cf. \cite{GPR}.

From the point of view of Mathematical Physics, Kaplitski{\u\i} \cite{kaplitski} has independently studied  the spectral properties of operators on the torus $\mathbb T^2$ with principal part
$$P=P_x\otimes P_y=\partial^2_{x,y},$$
obtaining for the counting function  estimates  
of type (\ref{prod5}). Reference therein is made to Arnold \cite{arnold}, suggesting to transfer the Weyl formula to hyperbolic equations. The results 
in \cite{kaplitski} can be essentially regarded as a particular case of those from \cite{batisti}.
Expansions of the type (\ref{prod5}) appear also in the recent paper of Coriasco and Maniccia \cite{CorMan} concerning the spectrum of the so-called SG-operators. 

Summing up, the results mentioned above cover the case of products of two operators,  $P=P_1\otimes P_2$, except for the computation of lower order terms in the expansions, cf. (\ref{prod8}). Thus, our attention will be mainly focus on the case $p\geq 3$ of (\ref{prod}) and lower order terms.

In the present paper, the attention will be rather addressed to results of (elementary) analytic number theory, which we shall present in Section \ref{counting functions} in detail; they are new by themselves, we believe. The applications to spectral theory will be given in the conclusive Section \ref{operators}. 
 We shall not construct here an algebra of (multisingular) pseudo-differential operators containing $P_1\otimes ...\otimes P_p$ for $p\geq 3$. Computations are cumbersome, involving a stratified calculus of the type of that from \cite{MR,nikolarodino2,schulze,schulze2}, occurring in other contexts. Finally, we shall limit ourselves instead to consider  perturbations of the type $P+Q$, where $Q$ is a lower order pseudo-differential operator.
\section{Asymptotics of some counting functions}\label{counting functions}
We study in this preparatory section the asymptotic behavior of some counting functions of ``multi-divisor'' type. They will be very helpful when applied to spectral asymptotics of various examples of ``multi-singular'' operators. 

\subsection{Counting functions of products of sequences.}\label{countinggeneral}
We start by considering the following general question. Let $\{\lambda^{(j)}_{k}\}_{k=1}^{\infty}$, $j=1,\dots,p$, be non-decreasing sequences of positive real numbers. The sequences are rather arbitrary and they are not necessarily linked to any operator.

Assuming that we have some knowledge about each of the counting functions
\begin{equation}
\label{cgeq1}
N_{j}(\lambda):=\sum_{\lambda^{(j)}_{k}\leq\lambda}1=\# \left\{k\in\mathbb{N}:\ \lambda^{(j)}_{k}\leq\lambda\right\}, \ \ \ j=1,\dots,p,
\end{equation}
we would like to obtain asymptotic information about the counting function of the $p$ products of the elements of the sequences, namely, 
\begin{equation}
\label{cgeq2}
N(\lambda):=\sum_{\lambda^{(1)}_{k_1}\lambda^{(2)}_{k_2}\dots\lambda^{(p)}_{k_p}\leq\lambda}1=\#\left\{(k_1,\dots,k_p)\in\mathbb{N}^{p}:\ \lambda^{(1)}_{k_1}\lambda^{(2)}_{k_2}\dots\lambda^{(p)}_{k_{p}}\leq\lambda\right\}.
\end{equation}

The next simple proposition tells us that it is always possible to find the asymptotic behavior of (\ref{cgeq2}) whenever there is a block of counting functions (\ref{cgeq1}) with dominating asymptotic behavior.

\begin{proposition}
\label{cgp1} Suppose that there are non-negative numbers $\tau<\alpha$ and indices $j_{1},\dots,j_{\nu}$, where $1\leq\nu\leq p$, such that 
\begin{equation}
\label{cgeq3}
N_{j_{q}}(\lambda)\sim A_{j_{q}}\lambda^{\alpha}, \ \ \ \lambda\to\infty, \ \ \ q=1,\dots,\nu,
\end{equation}
with $A_{j_{q}}\neq0$, and
\begin{equation}
\label{cgeq4}
N_{j}(\lambda)=O(\lambda^{\tau}), \ \ \ \lambda\to\infty, \ \ \ j\notin\left\{j_{1},\dots,j_{\nu}\right\}.
\end{equation}
Then, the counting function $(\ref{cgeq2})$ has asymptotic behavior
\begin{equation}
\label{cgeq5}
N(\lambda)\sim A \lambda^{\alpha}\frac{(\alpha \log \lambda)^{\nu-1}}{(\nu-1)!} , \ \ \ \lambda\to\infty,
\end{equation}
where
\begin{equation}
\label{cgeq6}
A= \left(\prod_{q=1}^{\nu}A_{j_{q}}\right)\cdot \left(\prod_{j\notin\left\{j_{1},\dots,j_{\nu}\right\}}\left(\sum_{k=1}^{\infty}\frac{1}{\left(\lambda^{(j)}_{k}\right)^{\alpha}}\right)\right).
\end{equation}
\end{proposition} 

We will divide the proof of Proposition \ref{cgp1} into two lemmas. The first lemma deals with the case in which all counting functions have asymptotic behavior of the same order.

\begin{lemma} \label{cgl1} If
$N_{j}(\lambda)\sim A_{j}\lambda^{\alpha},$ with $\alpha>0$ and $A_{j}\neq 0$, for $j=1,2,\dots,p$, then $(\ref{cgeq2})$ has asymptotics
$$
N(\lambda)\sim A \lambda^{\alpha}\frac{(\alpha \log \lambda)^{p-1}}{(p-1)!} , \ \ \ \lambda\to\infty,
$$
where $A=\prod_{j=1}^{p}A_{j}$.
\end{lemma}
\begin{proof} We proceed by induction. Assume that
$$
\tilde{N}(\lambda)=\sum_{\lambda^{(1)}_{k_1}\lambda^{(2)}_{k_2}\dots\lambda^{(p-1)}_{k_{p-1}}\leq\lambda}1 \sim \tilde{A}\lambda^{\alpha}\frac{(\alpha\log\lambda)^{p-2}}{(p-2)!},
$$ 
with $\tilde{A}=\prod_{j=1}^{p-1}A_{j}.$ We then have,
\begin{align*}
N(\lambda)&=\sum_{\lambda^{(p)}_{k}\leq\lambda} \tilde{N}(\lambda/\lambda^{(p)}_{k})
\\
&
=\frac{\tilde{A}\alpha^{p-2}}{(p-2)!}\sum_{\lambda^{(p)}_{k}\leq\lambda}(\lambda/\lambda_k^{(p)})^{\alpha}(\log(\lambda/\lambda^{(p)}_{k}))^{p-2}+ \sum_{\lambda^{(p)}_{k}\leq\sqrt{\lambda}}o((\lambda/\lambda_k^{(p)})^{\alpha}\log^{p-2}\lambda)
\\
&
\ \ \ +O(\lambda^{\alpha}\log^{p-2}\lambda)\cdot\sum_{\sqrt{\lambda}<\lambda^{(p)}_{k}\leq\lambda}\frac{1}{(\lambda^{(p)}_{k})^{\alpha}}
\\
& =\frac{\tilde{A}\alpha^{p-2}}{(p-2)!}\sum_{\lambda^{(p)}_{k}\leq\lambda}(\lambda/\lambda_k^{(p)})^{\alpha}(\log(\lambda/\lambda^{(p)}_{k}))^{p-2}+o(\lambda^{\alpha}\log^{p-1} \lambda)+O(\lambda^{\alpha}\log^{p-2}\lambda)
\\
&
=\frac{\tilde{A}\alpha^{p-2}}{(p-2)!}\int_{0}^{\lambda}(\lambda/t)^{\alpha}(\log(\lambda/t))^{p-2}dN_{p}(t)+o(\lambda^{\alpha}\log^{p-1} \lambda)
\\
&
=\frac{\tilde{A}\alpha^{p-1}}{(p-2)!}\lambda^{\alpha}\int_{0}^{\lambda}(\log(\lambda/t))^{p-2}\frac{N_{p}(t)}{t^{\alpha+1}}dt+o(\lambda^{\alpha}\log^{p-1} \lambda)
\\
&
=\frac{A_{p}\tilde{A}\alpha^{p-1}}{(p-2)!}\lambda^{\alpha}\sum_{j=0}^{p-2}\binom{p-2}{j}(-1)^{\nu}(\log\lambda)^{p-2-\nu}\int_{1}^{\lambda}\frac{(\log t)^{\nu}}{t}dt+o(\lambda^{\alpha}\log^{p-1} \lambda)
\\
&
\sim
A\lambda^{\alpha}\frac{(\alpha\log \lambda)^{p-1}}{(p-2)!}\sum_{j=0}^{p-2}\binom{p-2}{j}\frac{(-1)^{\nu}}{\nu+1}
\\
&=A\lambda^{\alpha}\frac{(\alpha\log \lambda)^{p-1}}{(p-2)!}\int_{0}^{1}(1-t)^{p-2}dt
=A\lambda^{\alpha}\frac{(\alpha\log \lambda)^{p-1}}{(p-1)!}.
\end{align*}
\end{proof}

We also have,
\begin{lemma}
\label{cgl2} If 
$$M_{1}(\lambda)=\sum_{\mu^{(1)}_{k}\leq \lambda}1=O(\lambda^{\tau}) \ \mbox{ and } \  M_{2}(\lambda)=\sum_{\mu^{(2)}_{k}\leq \lambda}1\sim B \lambda^{\alpha}\log^{b}\lambda, \ \ \ \lambda\to\infty$$
where $0\leq\tau<\alpha$, $B\neq 0$, and $b\geq 0$, then
$$
M(\lambda)=\sum_{\mu^{(1)}_k\mu^{(2)}_k\leq \lambda}1\sim\tilde{B}
\lambda^{\alpha}\log^{b}\lambda, \ \ \ \lambda\to\infty,
$$ 
where $\tilde{B}=B
\sum_{k=1}^{\infty}(\mu^{(1)}_k)^{-\alpha}$.
\end{lemma}
\begin{proof} Observe first that
$$
\sum_{k=1}^{\infty}\frac{1}{\left(\mu^{(1)}_{k}\right)^{\alpha}}= \int_{0}^{\infty}t^{-\alpha}dM_{1}(t)=\alpha\int_{0}^{\infty}\frac{M_1(t)}{t^{1+\alpha}}dt
$$
is convergent because $t^{-1-\alpha}M_1(t)=O(t^{-1-(\alpha-\tau)})$. Now,
\begin{align*}
M(\lambda)&=\sum_{\mu^{(1)}_{k}\leq\lambda}M_{2}(\lambda/\mu^{(1)}_{k})=\tilde{B}\lambda^{\alpha}\log^{b}\lambda-B\int_{0}^{\infty}\frac{\log^{b} t}{t^{\alpha}}dM_{1}(t)+o(\lambda^{\alpha}\log^{b}\lambda)
\\
&
=\tilde{B}\lambda^{\alpha}\log^{b}\lambda+O\left(\lambda^{\alpha}(\log^{b}\lambda) \int_{1}^{\lambda}\frac{dt}{t^{1+\alpha-t}}\right)+o(\lambda^{\alpha}\log^{b}\lambda)
\\
&
\sim \tilde{B}\lambda^{\alpha}\log^{b}\lambda,
\end{align*}
as claimed.
\end{proof}

We can now prove Proposition \ref{cgp1}.
\begin{proof}[Proof of Proposition \ref{cgp1}] Let $\{j_{\nu+1},\dots,j_{p}\}=\{1,2,\dots,p\}\setminus\{j_{1},\dots,j_{\nu}\}$
We arrange the two sequences of products
\begin{equation}
\label{cgeq7}
\prod_{q=\nu+1}^{p}\lambda^{(j_{q})}_{k_{j_{q}}}\ \ \mbox{ and }  \ \ 
\prod_{q=1}^{\nu}\lambda^{(j_{q})}_{k_{j_{q}}}
\end{equation}
in two non-decreasing sequences $\{\mu_{k}^{(1)}\}_{k=1}^{\infty}$ and $\{\mu_{k}^{(2)}\} _{k=1}^{\infty}$, respectively, where each element in these sequences is repeated as many times as it can be represented as in (\ref{cgeq7}). The hypothesis (\ref{cgeq3}) and Lemma \ref{cgl1} yield
$$
M_{2}(\lambda)=\sum_{\mu^{(2)}_{k}\leq \lambda}1\sim B \lambda^{\alpha} \log ^{\nu-1}\lambda , \ \ \ \lambda\to\infty,
$$
where $B=(\alpha^{\nu-1}/(\nu-1)!)\prod_{q=1}^{\nu}A_{j_{q}}$. On the other hand,  using (\ref{cgeq4}), one easily verifies that 
$$
M_{1}(\lambda)=\sum_{\mu^{(1)}_{k}\leq \lambda}1=O(\lambda^{\tau}\log^{p-\nu}\lambda), \ \ \ \lambda\rightarrow \infty.
$$ Applying Lemma \ref{cgl2} and noticing that
$$
\sum_{k=1}^\infty \frac{1}{\left(\mu^{(1)}_{k}\right)^{\alpha}}=\left(\prod_{j\notin\left\{j_{1},\dots,j_{\nu}\right\}}\left(\sum_{k=1}^{\infty}\frac{1}{\left(\lambda^{(j)}_{k}\right)^{\alpha}}\right)\right),
$$
we obtain the asymptotic formula (\ref{cgeq5}) with the constant (\ref{cgeq6}).
\end{proof}
\subsection{Remainders}\label{cgr} We now study the remainder in (\ref{cgeq5}). We impose stronger assumptions than (\ref{cgeq3}) on the leading counting functions.

We start by looking at the case when a single counting function dominates all others.

\begin{proposition} \label{cgp2} Assume that there are non-negative numbers $\tau<\eta<\alpha$ and an index $l\in\{1,\dots,p\}$ such that $N_{j}(\lambda)=O(\lambda^{\tau})$, for $j\neq l$, and $N_{l}$ satisfies
\begin{equation}
\label{cgeq8}
N_{l}(\lambda)= A_{l}\lambda^{\alpha}+O(\lambda^{\eta}), \ \ \ \lambda\to\infty,
\end{equation}
with $A_{l}\neq0$. Then, 
\begin{equation}
\label{cgeq9}
N(\lambda)= A \lambda^{\alpha}+O(\lambda^{\eta}) , \ \ \ \lambda\to\infty,
\end{equation}
where
\begin{equation}
\label{cgeq10}
A= A_{l}\prod_{j\neq l}\left(\sum_{k=1}^{\infty}\frac{1}{\left(\lambda^{(j)}_{k}\right)^{\alpha}}\right).
\end{equation}
\end{proposition}
\begin{proof} By renaming the sequences, we may assume that $l=1$. 
We use a recursive argument. Suppose that we have already established
$$
\tilde{N}(\lambda)=\sum_{\lambda^{(1)}_{k_1}\lambda^{(2)}_{k_2}\dots\lambda^{(p-1)}_{k_{p-1}}\leq\lambda}1= \tilde{A}\lambda^{\alpha}+O(\lambda^{\eta}).
$$ 
Since for any $b>\tau$
$$
\sum_{\lambda\leq \lambda^{(p)}_{k}}\frac{1}{(\lambda^{(p)}_{k})^{b}}=\frac{ N_{p}(\lambda)}{\lambda^{b}}+b\int_{\lambda}^{\infty}\frac{N_{p}(t)}{t^{b+1}}\:dt=O(\lambda^{\tau-b}), \ \ \ \lambda\to\infty,
$$
we have
\begin{align*}
N(\lambda)&=\sum_{\lambda^{(p)}_{k}\leq\lambda} \tilde{N}(\lambda/\lambda^{(p)}_{k})
\\
&
= \tilde{A}\lambda^{\alpha_1}\left(\sum_{k=1}^{\infty}\frac{1}{(\lambda^{(p)}_{k})^{\alpha}}\right)+O(\lambda^{\tau})+\sum_{\lambda^{(p)}_{k}\leq\lambda}O((\lambda/\lambda^{(p)}_{k})^{\eta})
\\
&
=A\lambda^{\alpha_1}+O(\lambda^{\tau})+O(\lambda^{\eta}),
\end{align*}
which shows (\ref{cgeq9}).
\end{proof}

For the analysis of the remaining case, we will employ a complex Tauberian theorem of Aramaki \cite{aramaki1988}.

\begin{proposition}
\label{cgp3} Assume there are non-negative numbers $\tau<\alpha$ and indices $j_{1},\dots,j_{\nu}$, where $1\leq\nu\leq p$, such that 
\begin{equation}
\label{cgeq11}
N_{j_{q}}(\lambda)= A_{j_{q}}\lambda^{\alpha}+O(\lambda^\tau), \ \ \ \lambda\to\infty, \ \ \ q=1,\dots,\nu,
\end{equation}
with $A_{j_{q}}\neq0$, and (\ref{cgeq4}) holds. Then, 
 there exists $\eta<\alpha $ such that $(\ref{cgeq2})$ has asymptotic expansion
\begin{align}
\label{cgeq12}
N(\lambda)&=\sum_{q=1}^{\nu} \frac{B_{q}}{(q-1)!}\left(\frac{d}{dz}\right)^{q-1}\left.\left(\frac{\lambda^{z}}{z}\right)\right|_{z=\alpha} +O(\lambda^{\eta})
\\
&
=\lambda^{\alpha}\left(A\frac{(\alpha\log \lambda)^{\nu-1}}{(\nu-1)!}+\sum_{q=0}^{\nu-2}C_{q}\log^{q}\lambda\right)+O(\lambda^{\eta})\nonumber
\end{align}
where $A$ is given by $(\ref{cgeq6})$ and the constants $B_{q}$ can be computed as
\begin{equation}
\label{cgeq13}
B_{q}= \frac{1}{(\nu-q)!}\left(\frac{d}{dz}\right)^{\nu-q}\left.\left((z-\alpha)^{\nu}\prod_{j=1}^{p}\sum_{k=1}^{\infty}\frac{1}{\left(\lambda^{(j)}_{k}\right)^{z}}\right)\right|_{z=\alpha}
\end{equation}
\end{proposition} 
\begin{proof}
Set 
$$
F_{j}(z)=\int_{0}^{\infty}t^{-z}dN_{j}(t)=\sum_{k=1}^{\infty}\frac{1}{\left(\lambda^{(j)}_{k}\right)^{z}}, \ \ \ j=1,\dots,p,
$$
and
$$
F(z)=\int_{0}^{\infty}t^{-z}dN(t).
$$
It is easy to verify that $F(z)$ and the $F_{j}(z)$ are analytic on the half-plane $\Re e\;z>\alpha$ and they are connected by the formula
$$
F(z)=\prod_{j=1}^{p}F_{j}(z).
$$
 Moreover, $F_{j}(z)$ is analytic on $\Re e\: z>\tau$ if $j\not\in\{j_{1},\dots,j_{q}\}$, whereas
$$
F_{j_{q}}(z)-\frac{A_{j_{q}}\alpha}{z-\alpha}, \ \ \ q=1,\dots,\nu,
$$
extend analytically to the same half-plane. Furthermore, these functions are of at most polynomial growth on any strip $\tau<a<\Re e\: z<b$. Thus, $F$ has also at most polynomial growth on such strips and it is meromorphic on $\Re e\: z>\tau$, having a single pole at $z=\alpha$ of order $\nu$. The hypotheses from Aramaki's Tauberian theorem are therefore satisfied, and the result follows at once from it.
\end{proof}
\subsection{Lower order terms in some special cases}  When the $\{\lambda^{(j)}_{k}\}_{k=1}^{\infty}$ arise as $\lambda^{(j)}_{k}=(c_{j}(k-1)+b_{j})^{\beta_{j}}$, where $c_{j}$, $b_{j}$, $\beta_{j}$ are positive constants, and one assumes $\beta_{p}>\beta_{p-1}>\dots>\beta_{1}> 0,$ it is possible to improve the asymptotic formula (\ref{cgeq9}) by giving lower order terms in the asymptotic expansion. Given $\vec {c}=(c_{1},c_{2},\dots,c_{p})$, $\vec {b}=(b_{1},b_{2},\dots,b_{p})$, $\vec {\beta}=(\beta_{1},\beta_{2},\dots,\beta_{p})\in\mathbb{R}_{+}^{p}$, we are interested in this subsection in the asymptotic behavior of the counting function 
\begin{equation}
\label{cseq1}
D^{\vec\beta}_{\vec c,\vec b}(\lambda)=\#\left\{(k_1,\dots,k_p)\in\mathbb{N}^{p}:\ (c_1k_1+b_1)^{\beta_1}(c_2k_2+b_2)^{\beta_2}\dots (c_{p}k_{p}+b_{p})^{\beta_{p}}\leq\lambda\right\}.
\end{equation}
For the constants in our expansions, we shall need the Hurwitz zeta function \cite[p. 251]{apostol}. It is defined for fixed $a>0$ as
\begin{equation}\label{cseq2}
\zeta(z;a)=\sum_{k=0}^{\infty}\frac{1}{(k+a)^{z}}, \ \ \ \Re e z>1.
\end{equation}
It is well-known that (\ref{cseq2}) admits meromorphic continuation to the whole complex plane, with a simple pole at $z=1$ with residue 1 (cf. \cite[p. 254]{apostol} or \cite[p. 348]{estrada-kanwal}). In particular, when $a=1$ we recover $\zeta(z)=\zeta(z,1)$, the Riemann zeta function. Using the Euler-Maclaurin summation formula \cite{apostol,estrada-kanwal,ivic}, one easily deduces the following asymptotic formula
\begin{equation}
\label{cseq3} 
\sum_{0\leq k\leq \lambda}\frac{1}{(k+a)^{s}}=\frac{(\lambda+a)^{1-s}}{1-s}+\zeta(s;a)+O(\lambda^{-s}), \ \ \ \lambda\to\infty, \mbox{ when }0<s \mbox{ and }s\neq 1.
\end{equation}

Observe that Proposition \ref{cgp1} immediately yields the dominant term in the asymptotic expansion of (\ref{cseq1}),
\begin{equation}
\label{cseq4} 
D^{\vec\beta}_{\vec c,\vec b}(\lambda) \sim \frac{1}{c_{1}}\left(\prod_{j=2}^{p}\frac{\zeta(\beta_{j}/\beta_{1};b_{j}/c_{j})}{c_{j}^{\beta_{j}/\beta_{1}}}\right) \lambda^{\frac{1}{\beta_{1}}} , \ \ \ \lambda\to\infty.
\end{equation}

We begin with the analysis of the case $p=2$. The proof of the following lemma is inspired by the classical Dirichlet hyperbola method \cite[p. 57]{apostol}.

\begin{lemma}
\label{lemma1}
Let $\vec \beta=(\beta_{1},\beta_{2})$ be such that $0<\beta_{1}<\beta_{2}$. Then,
\begin{equation}
\label{cseq5}
D^{\vec\beta}_{\vec c,\vec b}(\lambda)=\frac{\zeta(\beta_{2}/\beta_{1}; b_{2}/c_2)}{c_1c_{2}^{\beta_2/\beta_1}} \lambda^{\frac{1}{\beta_{1}}}+\frac{\zeta(\beta_{1}/\beta_{2}; b_{1}/c_1)}{c_2c_{1}^{\beta_1/\beta_2}} \lambda^{\frac{1}{\beta_{2}}}+O(\lambda^{\frac{1}{\beta_{1}+\beta_{2}}}).
\end{equation}
\end{lemma}
\begin{proof} Since $D^{\vec\beta}_{\vec c,\vec b}(\lambda)=D^{\vec \beta}_{\vec e,\vec d}(\lambda/(c_1^{\beta_1}c_2^{\beta_2}))$, where $\vec{e}=(1,1)$ and $\vec{d}=(b_1/c_1,b_2/c_2)$, we may assume that $c_1=c_2=1$. For ease of writing, we set $D^{\vec\beta}_{\vec b}=D^{\vec\beta}_{\vec e,\vec b}$.
We have that
\begin{align*}
D^{\vec\beta}_{\vec b}(\lambda)&=
\sum_{(k_1+b_1)^{\beta_1}(k_2+b_2)^{\beta_2}\leq \lambda} 1
\\
&
= \sum_{k_{1}+b_1\leq \lambda^{1/(\beta_{1}+\beta_{2})}} \left(\frac{\lambda}{(k_{1}+b_1)^{\beta_{1}}}\right)^{\frac{1}{\beta_{2}}}-k_{1}
+ \sum_{k_{2}+b_2\leq \lambda^{1/(\beta_{1}+\beta_{2})}} \left(\frac{x}{(k_{2}+b_2)^{\beta_{2}}}\right)^{\frac{1}{\beta_{1}}}-k_{2} + O(\lambda^{\frac{1}{\beta_{1}+\beta_{2}}})
\\
&
= \lambda^{\frac{1}{\beta_2}}I_{1,\beta_{1}/\beta_{2}}(\lambda^{1/(\beta_{1}+\beta_{2})}-b_1)+ \lambda^{\frac{1}{\beta_1}}I_{2,\beta_{2}/\beta_{1}}(\lambda^{1/(\beta_{1}+\beta_{2})}-b_2)- \lambda^{\frac{2}{\beta_{1}+\beta_{2}}}+O(\lambda^{\frac{1}{\beta_{1}+\beta_{2}}}),
\end{align*}
where $I_{j,s}(x)=\sum_{k\leq x} {(k+b_{j})^{-s}}$. The asymptotic formula (\ref{cseq3}) gives
\begin{align*}
\lambda^{\frac{1}{\beta_{2}}}I_{1,\beta_{1}/\beta_{2}}(\lambda^{1/(\beta_{1}+\beta_{2})}-b_1)&=\lambda^{\frac{1}{\beta_{2}}}\left(\zeta(\beta_{1}/\beta_{2};b_1)+\frac{\beta_2\lambda^{\frac{\beta_{2}-\beta_{1}}{\beta_2(\beta_{1}+\beta_{2})}}}{\beta_{1}-\beta_{2}}+O(\lambda^{\frac{-\beta_{1}}{\beta_{2}(\beta_{1}+\beta_{2})}})\right)
\\
&
=\lambda^{\frac{1}{\beta_{2}}}\zeta(\beta_{1}/\beta_{2};b_1)+\frac{\beta_2\lambda^{\frac{2}{\beta_{1}+\beta_{2}}}}{\beta_{2}-\beta_{1}}+O(\lambda^{\frac{1}{\beta_{1}+\beta_{2}}}),
\end{align*}
and similarly
$$
\lambda^{\frac{1}{\beta_{1}}}I_{2,\beta_{2}/\beta_{1}}(\lambda^{1/(\beta_{1}+\beta_{2})}-b_2)=\lambda^{\frac{1}{\beta_{1}}}\zeta(\beta_{2}/\beta_{1};b_2)+\frac{\beta_1\lambda^{\frac{2}{\beta_{1}+\beta_{2}}}}{\beta_{1}-\beta_{2}}+O(\lambda^{\frac{1}{\beta_{1}+\beta_{2}}}).
$$
The relation (\ref{cseq5}) follows on combining the three previous asymptotic formulas.
\end{proof}

In general, we have:

\begin{proposition}\label{v2}
Let $\beta=(\beta_{1},\dots,\beta_{p})\in\mathbb{R}^{p}$ be such that $\beta_{p}>\beta_{p-1}>\dots>\beta_{1}> 0.$
Then, the counting function (\ref{cseq1}) has asymptotics
\begin{equation}
\label{cseq6}
D^{\vec \beta}_{\vec c,\vec b}(\lambda)=\sum_{j=1}^{p}A_j\lambda^{\frac{1}{\beta_{j}}}+ O(\lambda^{\frac{p-1}{\beta_{1}+\dots+\beta_{p}}}), \ \ \ \lambda\to\infty,
\end{equation}
where $A_{j}=A_{j,\vec \beta,\vec c,\vec b} =c_{j}^{-1}\prod_{\nu\neq j}c^{- \beta_\nu/\beta_j}\zeta(\beta_{\nu}/\beta_{j};b_{\nu}/c_{\nu})$.
\end{proposition}
\begin{remark}
In (\ref{cseq6}), some of the terms may be absorbed by the error term, only those $j$ such that 
$$
(p-1)\beta_{j}\leq \beta_{1}+\dots+\beta_{p}
$$
occur in the sum. Of course, this always holds for $j=1,2$; thus, at least, we always have two leading terms in (\ref{cseq6}).
\end{remark}
\begin{proof} The case $p=2$ is Lemma \ref{lemma1}. Assume the result is valid for $p-1$, we proceed to show (\ref{cseq6}) by induction. As in Lemma \ref{lemma1}, we may suppose that $c_1=c_2=\dots=c_p=1$. For simplicity, we write $D^{\vec \beta}_{\vec b}=D^{\vec \beta}_{\vec c,\vec b}\ $. Set $\alpha=\sum_{j=2}^{p}\beta_{j}$, $\vec d=(b_2,b_3,\dots,b_{p})\in\mathbb{R}_{+}^{p-1}$, and $\vec {\eta}=(\beta_{2},\dots,\beta_{p})\in\mathbb{R}_{+}^{p-1}$. Write 
$$D^{\vec \beta}_{\vec b}(\lambda)=I_{1}(\lambda)+I_{2}(\lambda)+O(\lambda^{1/(\alpha+\beta_{1})}),$$ where
$$
I_{1}(\lambda)=\sum_{k_{1}+b_1\leq \lambda^{1/(\alpha+\beta_{1})}}D^{\vec{\eta}}_{\vec d}\left(\lambda/(k_{1}+b_1)^{\beta_{1}}\right),
$$
$$
I_{2}(\lambda)=\sum_{\prod_{j=2}^{p}(k_{j}+b_j)^{\beta_{j}}\leq \lambda^{\alpha/(\alpha+\beta_{1})}} \left(\frac{\lambda}{(k_{2}+b_2)^{\beta_{2}}(k_{3}+b_3)^{\beta_3}\dots (k_{p}+b_p)^{\beta_{p}}}\right)^{1/\beta_{1}}- \lambda^{\frac{1}{\alpha+\beta_{1}}}D^{\vec{\eta}}_{\vec d}( \lambda^{\alpha/(\alpha+\beta_{1})}),
$$
and 
\begin{align*}
D^{\vec{\eta}}_{\vec d}(\lambda)&=\#\left\{(k_2,\dots,k_{p})\in\mathbb{N}^{p-1}:\ (k_2+b_2)^{\beta_2}\dots (k_{p}+b_{p})^{\beta_{p}}\leq\lambda\right\}
\\
&
=\sum_{j=2}^{p}\tilde{A}_j\lambda^{\frac{1}{\beta_{j}}}+ O(\lambda^{\frac{p-2}{\alpha}}), \ \ \ \lambda\to\infty,
\end{align*}
with $\tilde{A}_{j}=\prod_{2\leq \nu,\nu\neq j}\zeta(\beta_{\nu}/\beta_{j};b_{\nu})$ for $j=2,\dots, p$. If we combine the latter with (\ref{cseq3}), we conclude that the asymptotic behavior of  $I_{1}(\lambda)$ is
\begin{align*}
I_{1}(\lambda)&=\sum_{j=2}^{p}\tilde{A}_{j} \sum_{k+b_1\leq \lambda^{1/(\alpha+\beta_{1})}}\frac{\lambda^{1/\beta_{j}}}{(k+b_1)^{\beta_{1}/\beta_{j}}} +O(\lambda^{(p-1)/(\alpha+\beta_{1})})
\\
&
=\sum_{j=2}^{p}\tilde{A}_{j}\left(\zeta\left(\beta_{1}/\beta_{j};b_1\right)\lambda^{1/\beta_{j}}+\frac{\beta_{j}\lambda^{(\alpha+\beta_{j})/(\beta_{j}(\alpha+\beta_{1}))}}{\beta_{j}-\beta_{1}}\right)+O(\lambda^{(p-1)/(\alpha+\beta_{1})})
\\
&
=\sum_{j=2}^{p}A_{j}\lambda^{1/\beta_{j}}+\tilde{A}_{j}\frac{\beta_{j}\lambda^{(\alpha+\beta_{j})/(\beta_{j}(\alpha+\beta_{1}))}}{\beta_{j}-\beta_{1}}+O(\lambda^{(p-1)/(\alpha+\beta_{1})}).
\end{align*}
Observe that $C:=\beta_{1}^{-1}\int_{0}^{\infty}t^{-1-1/\beta_{1}}D^{\vec{\eta}}_{\vec d}(t)dt$ is absolutely convergent. We then have
\begin{align*}
I_{2}(\lambda)&= \lambda^{1/\beta_{1}}\int_{0}^{\lambda^{\alpha/(\alpha+\beta_{1})}}t^{-1/\beta_{1}}dD^{\vec{\eta}}_{\vec d}(t)-\lambda^{\frac{1}{\alpha+\beta_{1}}}D^{\vec{\eta}}_{\vec d}( \lambda^{\alpha/(\alpha+\beta_{1})})
\\
&=\frac{\lambda^{1/\beta_{1}}}{\beta_{1}}\int_{0}^{\lambda^{\alpha/(\alpha+\beta_{1})}}t^{-1-1/\beta_{1}}D^{\vec{\eta}}_{\vec d}(t)dt
\\
&=C\lambda^{1/\beta_{1}}- \frac{\lambda^{1/\beta_{1}}}{\beta_{1}}\int_{\lambda^{\alpha/(\alpha+\beta_{1})}}^{\infty}t^{-1-1/\beta_{1}}D^{\vec{\eta}}_{\vec d}(t)dt
\\
&= C\lambda^{1/\beta_{1}}-\sum_{j=2}^{p}\frac{\tilde{A}_{j}\beta_{j}\lambda^{(\alpha+\beta_{j})/(\beta_{j}(\alpha+\beta_{1}))}}{\beta_{j}-\beta_{1}}+O(\lambda^{(p-1)/(\alpha+\beta_{1})}).
\end{align*}
Thus, we have shown (\ref{cseq6}) except for $C=\prod_{\nu=2}^{p}\zeta(\beta_{\nu}/\beta_{1};b_\nu)$. But this fact follows by comparison with (\ref{cseq4}). The proof is complete.
\end{proof}
\begin{remark}
In connection with Proposition \ref{v2}, Estrada and Kanwal have given an interesting distributional treatment of the asymptotic expansions of type (\ref{cseq6}), which often leads to improvements in the error term when interpreted in the distributional sense (cf. \cite[Sec. 5.3]{estrada-kanwal}).
\end{remark}

\section{Counting functions for tensor products of pseudo-differential operators}
\label{operators}
We now apply results of Section \ref{counting functions}
to the spectral asymptotics of the tensor products of pseudo-differential operators, and 
their perturbations. We shall mainly refer to operators in the Euclidean setting. 
Parallel results for operators on compact manifold will be outlined at the end. For the sake of completeness, we begin  with a short survey of the 
classes of M. Shubin, cf. \cite{BBR,helfer,NR,shubin}.

\subsection{Globally elliptic pseudo-differential operators}
Write $z=(x,\xi)\in\mathbb R^{2n}$  and $<z>=(1+|z|^2)^{1/2}=(1+|x|^2+|\xi|^2)^{1/2}.$ One defines 
the class of symbols $\Gamma^m_{\rho}(\mathbb{R}^{n})$, $m\in\mathbb R,$ $0<\rho\leq 1,$  as the set of all functions $a\in C^{\infty}(\mathbb R^{2n})$ satisfying, for all $\gamma,$
\begin{equation}
\label{eq31}
|\partial^\gamma_za(z)|\leq C_\gamma<z>^{m-\rho|\gamma|}, \ \ \ z\in\mathbb R^{2n},
\end{equation}
with constants independent of $z.$ The corresponding pseudo-differential operator is defined by Weyl quantization as
\begin{equation}
\label{eq32}
Pu(x)=a^wu(x)=\frac{1}{(2\pi)^{n}}\int e^{i(x-y)\xi}a\left(\frac{x+y}{2},\xi\right)u(y)dyd\xi.
\end{equation}
Note that if the symbol $a$ is a polynomial in the $\xi$ variables, i.e. $P$ in (\ref{eq32}) is a partial differential operator, then the estimates (\ref{eq31}) force $a(z)$ to be a polynomial in the $x-$variables as well, i.e. $P$ is a partial differential operator with polynomial coefficients.

Let us introduce the global Sobolev spaces $H^s(\mathbb R^{n}), s\in\mathbb N,$ Hilbert spaces with the norm
\begin{equation}
\label{eq33}
||u||_s=\sum_{|\alpha|+|\beta|\leq s}||x^\alpha D^\beta u||<\infty.
\end{equation} 
By interpolation and duality  the definition extends to $s\in\mathbb R$, and we have $\bigcap_sH^s(\mathbb{R}^{n})=\mathcal S(\mathbb R^n),$  $\bigcup_sH^s(\mathbb{R}^{n})=
\mathcal S'(\mathbb R^n)$. The immersion $\iota:^s H^s\rightarrow H^t$ is compact for $s>t.$ If $a\in\Gamma^m_\rho(\mathbb R^n),$
then $a^w:H^s(\mathbb R^n)\rightarrow H^{s-m}(\mathbb R^n)$ continuously for every $s\in\mathbb R,$ hence 
$a^w: \mathcal S(\mathbb R^n)\rightarrow \mathcal S(\mathbb R^n),$ $\mathcal S'(\mathbb R^n)\rightarrow \mathcal S'(\mathbb R^n).$
In the following we shall assume that for large $|z|$,
\begin{equation}
\label{eq34}
a(z)=a_m(z)+a_{m-\rho}(z),
\end{equation}
where $a_m(tz)=t^ma_m(z), t>0.$ We then say that $a$ is globally elliptic if
\begin{equation}
\label{eq35}
a_m(z)\neq 0 \ \ \ \mbox{ for } z\neq 0.
\end{equation}
Operators with globally elliptic symbol possess parametrix. Namely, there exists $b\in\Gamma^{-m}_\rho(\mathbb R^n)$ such that $a^wb^w=I+R_1$ and 
$b^wa^w=I+R_2$, where $R_1,R_2:$ $\mathcal S'(\mathbb R^n)\rightarrow \mathcal S(\mathbb R^n).$ It follows that $a^w:H^s(\mathbb R^n)\rightarrow H^{s-m}(\mathbb R^n)$
is a Fredholm operator and then eigenfunctions, i.e. solutions of $a^wu=0,$ do not depend on $s\in\mathbb R$ and belong to $\mathcal S(\mathbb R^n).$
Passing now to spectral theory, we assume that $a\in\Gamma^{m}_\rho(\mathbb R^n),$ $m>0,$
is real-valued and globally elliptic with $a_m(z)>0$, for $z\neq 0.$ Then $P=a^wu: H^{m}(\mathbb R^n)\mapsto L^2(\mathbb R^n)$ is self-adjoint. The resolvent is compact and the spectrum is given by a sequence of real eigenvalues $\lambda_k\rightarrow \infty$ with finite multiplicity; the eigenfunctions belong to  
 $\mathcal S(\mathbb R)$ and form an orthonormal basis. The spectral counting function $N_P(\lambda)=\#\{k: \lambda_k\leq \lambda\}$ behaves as
 \begin{equation}
 \label{eq36}
 N_P(\lambda)=A\lambda^{2n/m}+O(\lambda^\sigma), \ \ \ \lambda\rightarrow \infty,
 \end{equation}
for some $\sigma<2n/m,$ with 
\begin{equation}
 \label{eq37}
 A=\frac{1}{(2\pi)^{n}}\int_{a_m(z)\leq 1}dz.
 \end{equation}
A sharp form of the remainder in (\ref{eq36}) can be obtained when $a\in\Gamma^m(\mathbb R^n)=\Gamma_1^m(\mathbb R^n)$
admits an asymptotic expansion in homogeneous terms $a\sim \sum_{k\in\mathbb N}a_{m-2k}.$ Then, with $A$ as before,
\begin{equation}
\label{eq38}
N_P(\lambda)=A\lambda^{2n/m}+O(\lambda^{2(n-1)/m}),
\end{equation}
see, for example, Helffer \cite[p. 175]{helfer}. In the sequel, we shall  assume that $P$
is strictly positive, so that $0< \lambda_1\leq \lambda_2\leq \dots$.
For $P$
as before, we may define the complex powers $P^z, z\in\mathbb C.$ They are trace class operators if $\Re e z<-2n/m,$ 
and, by analytic continuation, we define the zeta function associated to $P$ as
\begin{equation}
\label{eq39}
\zeta_{P}(z)= \operatorname*{Tr}(P^{-z})=\sum_{k=1}^\infty\lambda_k^{-z}.
\end{equation}
\subsection{Spectral asymptotics  for tensor products}
To give a precise functional frame to the results in the sequel, we shall introduce first  the tensorized global Sobolev spaces.
Write now $x_j,y_j\in\mathbb R^{n_j},$ $z_{j}=(x_j,y_j)\in\mathbb R^{2n_j},$ $j=1,...,p,$ $n=n_1+...+n_p,$ $x=(x_1,...,x_p),$ $ y=(y_1,...,y_p)\in\mathbb R^p,$
$z=(z_1,...,z_p)=(x_1,y_1,...,x_p,y_p)$. For $\vec{s}=(s_1,...,s_p)\in\mathbb R^p$, we define the tensor product of Hilbert spaces
\begin{equation}
\label{eq310}
H^{\vec{s}}(\mathbb R^n)=\bigotimes_{j=1}^pH^{s_j}(\mathbb R^{n_j}).
\end{equation}
When the components of $\vec{s}$ are non-negative integers, from (\ref{eq33}) we recapture as norm
\begin{equation}
\label{eq311}
||u||_{\vec{s}}=\underset{j=1,...,p}{\sum_{|\alpha_j|+|\beta_j|\leq s_j}}||x^{\alpha_1}...x^{\alpha_2}D^{\beta_1}_{x_1}...D^{\beta_p}_{x_p}u||.
\end{equation}
We have $\bigcap_{\vec{s}}H^{\vec{s}}(\mathbb R^n)=\mathcal S(\mathbb R^n)$ and $\bigcup_{\vec{s}}H^{\vec{s}}(\mathbb{R}^{n})=\mathcal S'(\mathbb R^n).$
The immersion $\iota: H^{\vec{s}}(\mathbb{R}^{n})\rightarrow H^{\vec{t}}(\mathbb{R}^{n})$ is compact if $\vec{s}>\vec{t},$ i.e. $s_j>t_j$ for $j=1,...,p.$

As announced at the Introduction, we consider now $P_j=a_{j}^w$ in $\mathbb R^{n_j},$ $j=1,...,p,$ with real-valued symbol $a_{j}\in\Gamma^{m_j}(\mathbb R^{n_j}),$ $m_j>0,$,
 and $a_{m_j}(z)>0$ for $z\neq 0$ in (\ref{eq35}); we further define 
\begin{equation}
\label{eq312}
P=P_1\otimes...\otimes P_p,
\end{equation}
as operator $P: H^{\vec{s}}(\mathbb R^n)\rightarrow H^{\vec{s}-\vec{m}}(\mathbb R^n),$ $\vec{m}=(m_1,...,m_p), $ for every  $\vec{s}\in\mathbb R^p.$ In particular, we have $P: H^{\vec{m}}(\mathbb R^n)\rightarrow L^{2}(\mathbb R^n)$ and 
$P: \mathcal S(\mathbb R^n) \rightarrow \mathcal S(\mathbb R^n),\; \mathcal S'(\mathbb R^n)\rightarrow \mathcal S'(\mathbb R^n).$
Moreover, $P$ is self-adjoint and strictly positive, if the factors $P_j$ are assumed to be strictly positive. 

If we denote by $\{\lambda_k^{(j)}\}_{k=1}^\infty$ the eigenvalues of $P_j$, according to the Introduction, the eigenvalues of $P$ are of the form 
$\lambda^{(1)}_{k_1}...\lambda^{(p)}_{k_p}$ and the eigenfunctions are tensor products of the respective eigenfunctions, hence they belong to $\mathcal S(\mathbb R^n).$

It is worth observing that $P$ can be written in the pseudo-differential form (\ref{eq32}) with symbol $a(z)=a_1(z_1)...a_p(z_p).$ However, the estimate (\ref{eq31})
fails in general, and the considerations of Subsection 3.1 do not apply in this context. For the case $p=2$, we address to \cite{batisti},  cf. \cite{BaGrPiRo,nikolarodino2}, where a calculus was achieved in terms of vector-valued symbols. Here, to find an asymptotic expansion for $N_P(\lambda)$, we shall use (\ref{prod2}) in combination with the analysis of Section \ref{counting functions}. In fact, from (\ref{eq36}) and (\ref{eq37}), we have
\begin{equation}\label{eq313}
N_{P_j}\sim A_j\lambda^{2n_j/m_j} \ \ \ \mbox{ with } A_j=\frac{1}{(2\pi)^{n_{j}}}\int_{a_{m_{j}}(z_{j})\leq 1}dz_j.
\end{equation}
Writing $\zeta_{P_j}$ for the zeta function of $P_j$, we inmediately obtain from Proposition \ref{cgp1}:
\begin{theorem}\label{ellipth1}
Let $P=P_1\otimes\dots\otimes P_p$ be as above and let $\alpha=\max_{j}\{2n_{j}/m_{j}\}$. Let further $j_{1},\dots,j_{\nu}$ be the indices such that $\alpha=2n_{j_{q}}/m_{j_{q}}$, $q=1,\dots,\nu$. Then, $P$ has spectral asymptotics
\begin{equation}
\label{eq314}
N_{P}(\lambda)=\sum_{\lambda^{(1)}_{k_1}\lambda^{(2)}_{k_2}\dots\lambda^{(p)}_{k_p}\leq\lambda}1\sim\left(\prod_{q=1}^{\nu}A_{j_{q}}\cdot \prod_{j\notin\left\{j_{1},\dots,j_{q}\right\}}\zeta_{P_{j}}(\alpha)\right) \lambda^{\alpha}\frac{(\alpha\log\lambda)^{\nu-1}}{(\nu-1)!}, \ \ \ \lambda\to\infty,
\end{equation}
where $A_{j}$ is given by $(\ref{eq313})$.
\end{theorem}
We remark that the case $p=2,$ $\nu=1$ or $\nu=2$, of Theorem \ref{ellipth1} also follows from the results of \cite{BaGrPiRo}, see also \cite{batisti}). 

As far as the reminder in (\ref{eq314}) concerns, from
(\ref{eq36}) and Proposition \ref{cgp3}, we obtain
\begin{equation}
\label{eq315}
N_P(\lambda)=\lambda^\alpha\sum_{q=0}^{\nu-1}
C_q\log^q\lambda+O(\lambda^\eta),
\end{equation}
for some $\eta<\alpha.$ The coefficient $C_{\nu-1}$ is given by (\ref{eq314}) and the other constants $C_q,$  $q=0,...,\nu-2,$ are determined by $(\ref{cgeq12})$,
$(\ref{cgeq13})$, and the values of the derivatives or poles of the zeta functions $\zeta_{P_{j}}(z)$ at $z=\alpha$, $j=1,\dots,p.$

Willing sharp values
of $\eta$ in the remainder, we further assume that $a_{j}\in\Gamma^{m_j}(\mathbb R^{n_j})$ with $a_{j}\sim\sum_{k\in\mathbb N}a_{m_j-2k}$ and we use (\ref{eq38}). Proposition  \ref{cgp2} yields,
\begin{theorem}\label{th32}
Let $P=P_1\times ...\otimes P_p$ be as above. Assume that there is an index $l\in\{1,...,p\}$ such that $2n_l/m_l>\beta=\max_{j\neq l}\{2n_j/m_j\}$. Then
\begin{equation}
\label{eq316}
N_P(\lambda)=\left(A_{l}\prod_{j\neq l}\zeta_{P_j}(\alpha)\right)\lambda^{\frac{2n_{l}}{m_{l}}}+O(\lambda^\eta),
\end{equation}
for any $\eta$ with $\eta>\max\{\beta, 2(n_l-1)/m_l\}.$
\end{theorem}
The following example shows that the exponent $\eta=\beta$ is sharp in (\ref{eq316}).
\begin{example}[Tensorized Hermite operators]\label{ex33} For tensor products of Hermite operators it is possible to detect lower 
order terms in the asymptotic expansion (\ref{eq316}). Namely, let us fix $\vec{\beta}=(\beta_1,...,\beta_p)$ with $\beta_1<...<\beta_p, \vec{c}=(c_1,...,c_p),
\vec{b}=(b_1,...,b_p)$, $p$-tuples of positive real numbers, cf. Subsection 2.3, and consider 
\begin{equation}
\label{eq318}
H_{j,c_j,b_j}=\frac{c_j}{2}(-\partial^2_{x_j}+x_j^2)-\frac{c_j}{2}+b_j, \ \ \ j=1,...,p,
\end{equation}
so that for $c_j=1$, $b_j=1$, we recapture $H_j$ in (\ref{prod3}) of the Introduction. The eigenvalues of $H_{j,c_j,b_j},$ as one dimensional operator, are $\lambda_k^{(j)}=c_j(k-1)+b_j,$ $k=1,2,...$. We then define the tensorized Hermite operator 
\begin{equation}
\label{eq319}
H^{\vec{\beta}}_{\vec{c},\vec{b}}=\otimes_{j=1}^pH^{\beta_j}_{j,c_j,b_j}.
\end{equation}
By Proposition \ref{v2}, we have for the corresponding counting function
\begin{equation}
\label{eq320}
N(\lambda)=D^{\vec{\beta}}_{\vec{c},\vec{b}}(\lambda)=\sum_{j=1}^pA_j\lambda^{1/\beta_j}+O(\lambda^{\frac{p-1}{\beta_1+...+\beta_p}})
\end{equation}
with $A_j$ as in Proposition \ref{v2}. In particular, for $p=2, c_j=1, b_j=1, j=1,2,$ we obtain (\ref{prod8}) of the Introduction.
\end{example}
\subsection{Asymptotics for lower order perturbations}
For simplicity, we shall assume that the factors $P_j$ in $P=P_1\otimes...\otimes P_p$ are partial differential operators with polynomial coefficients:
\begin{equation}
\label{eq321}
P_j=\sum_{|\alpha_j|+|\beta_j|\leq m_j}c^{(j)}_{\alpha_j,\beta_j}x^{\alpha_j}D_{x_j}^{\beta_j},\ \ \  x_j\in \mathbb R^{n_j}.
\end{equation}
As before, we assume that $P_j$ is elliptic, with principal symbol
\begin{equation}
\label{eq322}
p_{m_j}^{(j)}(x,\xi)=\sum_{|\alpha_j|+|\beta_j|=m_j}c^{(j)}_{\alpha_j,\beta_j}x_j^{\alpha_j}\xi_{j}^{\beta_j}>0 \ \ \ \mbox{for }  (x_j,\xi_j)\neq
(0,0),
\end{equation}
self-adjoint and strictly positive, $j=1,...,p$. We shall study 
\begin{equation}
\label{eq323}
A=P+R,
\end{equation}
where  $R$ is a partial differential operator with polynomial coefficients having lower order with respect to $P$,
in the sense that, writing $\vec\alpha=(\alpha_1,...,\alpha_p), \vec\beta=(\beta_1,...,\beta_p)\in\mathbb N^n, n=n_1+...+n_p$,
\begin{equation}
\label{eq324}
R=\underset{j=1,...,p}{\sum_{|\alpha_j|+|\beta_j|< m_j}}
c_{\alpha \beta}x^{\vec\alpha} D^{\vec\beta}.
\end{equation}
Note that each term of the sum in (\ref{eq324}) can be regarded as a tensor product:
$$x^{\vec\alpha} D^{\vec\beta}=x_1^{\alpha_1}D^{\beta_1}_{x_1}\otimes...\otimes x_p^{\alpha_p}D^{\beta_p}_{x_p},
$$
hence for every $\vec{s}\in \mathbb R^p$,
$$A=P+R: \;H^{\vec{s}}(\mathbb R^n)
\rightarrow H^{\vec{s}-\vec{m}}(\mathbb R^n).
$$
We shall first construct a parametrix for $A$. In absence of symbolic calculus, we shall use in the proof a direct argument.
\begin{proposition}\label{prop31}
For every fixed integer $M>0$, we can find $B: H^{\vec{s}}(\mathbb R^n)\rightarrow H^{\vec{s}+\vec{m}}(\mathbb R^n)$ for every $\vec{s}=(s_1,...,s_p)\in \mathbb R^p$, such that $BA=I+S', AB=I+S'',$ where $S', S'': H^{\vec{s}}(\mathbb R^n)\rightarrow H^{\vec{s}+\vec{M}}(\mathbb R^n),$ with $\vec{M}=(M,...,M).$
\end{proposition} 
\begin{proof}
Consider $$P^{-1}=P_1^{-1}\otimes...\otimes P_p^{-1}: H^{\vec{s}}(\mathbb R^n)\rightarrow H^{\vec{s}+\vec{m}}(\mathbb R^n).$$ 
We have $P^{-1}A=P^{-1}(P+R)=I-S$ with $$S=-P^{-1}R:H^{\vec{s}}(\mathbb R^n)\rightarrow H^{\vec{s}+\vec{1}}(\mathbb R^n).$$ Define then
$$B=\sum_{j=0}^{M-1}S^jP^{-1}: H^{\vec{s}}(\mathbb R^n)\rightarrow H^{\vec{s}+\vec{M}}(\mathbb R^n).
$$
We have
$$BA=\sum_{j=0}^{M-1}S_{j}(I-S)=I-S^M,
$$
where 
$S'=-S^M:H^{\vec{s}}(\mathbb R^n)\rightarrow H^{\vec{s}+\vec{M}}(\mathbb R^n)$. 
It is easy to check that $B$ is also a right parametrix.
\end{proof}
\begin{corollary}
\label{cor31}
The solution $u\in\mathcal S'(\mathbb R^n)$ of $Au=f\in\mathcal S(\mathbb R^n)$
belongs to $\mathcal S(\mathbb R^n)$
\end{corollary}
\begin{proof}
If $u\in\mathcal S'(\mathbb R^n)$, then $u\in H^{\vec{s}}(\mathbb R^n)$ for some $\vec{s}$. Taking $B$ as in Proposition 
\ref{prop31}, we obtain
$$
BAu=(I+S')u=Bf,
$$
hence, $u=Bf-S'u.$ We have $Bf\in\mathcal S(\mathbb R^n)$ and $S'u\in H^{\vec{s}+\vec{M}}(\mathbb R^n).$ Since $M$ in Proposition \ref{prop31}
can be fixed as large as we want, we conclude $u\in\mathcal S(\mathbb R^n).$
\end{proof}
\begin{corollary}\label{cor32}
The operator $A: H^{\vec{s}}(\mathbb R^n)\rightarrow H^{\vec{s}-\vec{m}}(\mathbb R^n)$ is Fredholm, for every fixed 
$\vec{s}\in\mathbb R^n.$
\end{corollary}
\begin{proof}
Let us apply Proposition \ref{prop31}  with $M=1.$ Since the inclusion
$ H^{\vec{s}+\vec{1}}(\mathbb R^n)\rightarrow H^{\vec{s}}(\mathbb R^n)$ is compact, the  Fredholm property is proved.
\end{proof}
Let us assume now that the operator $A$
in (\ref{eq323}) is self-adjoint. It follows from the preceding arguments that  the resolvent is compact and the eigenfunctions belong to
$\mathcal S(\mathbb R^n).$ Assume further that $A$ is strictly positive; we write $0<\lambda_1\leq\lambda_2\leq...$ for its eigenvalues and $N_{A}$ for its spectral counting function.
We give below an asymptotic formula for $\lambda_k$. In the sequel we write $\asymp$ to mean that $f=O(g)$ and $g=O(f)$
are both valid.
\begin{theorem}\label{th37}
Let $A=P+R$ in $(\ref{eq323})$ be as above. We use for $P$ the notation of Theorem \ref{ellipth1}, namely we write $\alpha=\max_{j}\{2n_j/m_j\}$
and we assume that $\alpha=2n_j/m_j$ for $\nu$ indices. We then have
\begin{equation}
\label{eq325}
\lambda_k\asymp k^{1/\alpha}(\log k)^{-(\nu-1)/\alpha},\ \ \  k\rightarrow \infty
\end{equation} 
and
\begin{equation}
\label{eq3251}
N_{A}(\lambda)\asymp \lambda^{\alpha}\log^{\nu-1} \lambda,\ \ \  \lambda\rightarrow \infty.
\end{equation}
\end{theorem}
\begin{proof}
We have
$$||Au||^{2}=||AP^{-1}Pu||^{2}\leq C_1||Pu||^2
\mbox{ with } C_1=||AP^{-1}||^2_{{\mathcal L}(L^2)}.$$
 On the other hand, using Proposition \ref{prop31}, we may write $I=BA-S'$ and thus
$$||Pu||^2=||P(BA-S')u||^2\leq 2||PBAu||^2+2||PS'u||^2|\leq C_2(||Au||^2+||u||^2)
$$
with
$$C_2=2\max \{||PB||^2_{\mathcal L(L^2)}, ||PS'||^2_{\mathcal L(L^2)}\},
$$
where $||PS'||_{\mathcal L(L^2)}<\infty$ if $M$ in Proposition \ref{prop31} is chosen sufficiently large. We now rewrite the preceding estimates as
$$( A^2u,u)\leq( C_1P^2u,u) \; \mbox{ and }\; (P^2u,u) \leq(C_2(A^2+I)u,u).
$$
Using the classical max-min formula for the eigenvalues of $A^2, P^2$
and denoting here $\mu_k$ the eigenvalues of $P$, we deduce
$$\lambda^2_k\leq C_1\mu_k^2 \mbox{ and } \mu_k^2\leq C_2(\lambda_k^2+1).
$$
Hence $\lambda_k\asymp\mu_k$. As a final step in the proof, we apply the following lemma.
\begin{lemma}
\label{le33}
If the sequence $0<\mu_1\leq \mu_2\leq...$, $\mu_k\rightarrow \infty$, admits counting function
$$N(\mu)\sim r\mu^\alpha\log^s\mu, \ \ \ \mu\rightarrow \infty,
$$
with $r,\alpha>0$ and $s\geq 0,$ then
\begin{equation}
\label{eq326}
\mu_k\sim\left(\frac{\alpha}{r}\right)^{1/\alpha}k^{1/\alpha}(\log k)^{-s/\alpha}, \ \ \ k\rightarrow \infty.
\end{equation}
\end{lemma}
The proof of this lemma is a simple combination of Proposition 4.6.4, page 198, and Lemma 5.2.9, page 219, from \cite{NR} and it is therefore omitted. Since for the counting function $N_P(\mu)$
we have from Theorem \ref{ellipth1}
$$N_P(\mu)\sim r\mu^\alpha(\log \mu)^{\nu-1}
$$
for a constant $r$, we deduce from Lemma \ref{le33} for the eigenvalues $\mu_k$ of $P$
the asymptotics (\ref{eq326}) with $s=\nu-1$. Hence (\ref{eq325}) follows. The asymptotic formula (\ref{eq3251}) can be easily deduced from (\ref{eq325}), we leave details to the reader.
\end{proof}
The rough asymptotics (\ref{eq3251}) can hopefully be improved, as suggested by the result from \cite{BaGrPiRo}, which gives
$N_A(\lambda)\sim N_P(\lambda)$ in the case $p=2.$ Furthermore, we expect formula (\ref{eq315}) is invariant under lower order perturbations. On the contrary, the precise asymptotics (\ref{eq320})
for tensorized  Hermite operators should be lost, after addition of lower order terms.

\subsection{Pseudo-differential operators on closed manifolds}We now look at pseu\-do-differential operators on closed manifolds.
Let $M_{1},M_{2},\dots, M_{p}$ be closed manifolds with $\dim M_{j}=n_{j}$. We consider elliptic self-adjoint pseudo-differential operators $P_{j}$ on $M_{j}$ of order $m_{j}$ and  principal symbol $a_{m_j}(x_{j},\xi_{j})>0$ for $(x_{j},\xi_{j})\in T^{\ast}M_{j}\setminus(M_{j}\times \left\{0\right\}), j=1,...,p$. We denote by $dx_{j} d\xi_{j}$ the natural volume 
form on the cotangent bundle $T^{\ast}M_{j}$. Under these circumstances, H\"{o}rmander's theorem \cite[Chap. III]{shubin} gives us the asymptotic behavior of each counting function $N_{P_{j}}(\lambda)$ of the eigenvalues $\{\lambda^{(j)}_{k}\}_{k=1}^{\infty}$ of $P_{j}$. In fact,
\begin{equation*}
N_{P_{j}}(\lambda)=\sum_{\lambda^{(j)}_{k}\leq \lambda}1= A_{j}\lambda^{n_{j}/m_{j}} + O(\lambda ^{(n_{j}-1)/m_{j}}), \ \ \ \lambda\to\infty,
\end{equation*}
where
\begin{equation}
\label{ellipeq4}
A_{j}=\frac{1}{(2\pi)^{n_{j}}}\int_{a_{m_j}(x_{j},\xi_{j})<1}dx_{j}d\xi_{j}, \ \ \ j=1,\dots,p.
\end{equation}
As usual, $\zeta_{P_{j}}$ denotes the zeta function of the operator $P_{j}$. 

Proposition \ref{cgp2} directly gives the spectral asymptotics of the operator $P=P_{1}\otimes P_{2}\otimes \dots \otimes P_{p}$ on the closed manifold $M=M_{1}\times M_{2}\times \dots \times M_{p}\:$ of dimension $\dim M=n=n_{1}+n_{2}+\dots+n_{p}$, whenever one of the counting functions $N_{P_{l}}$ dominates all the others.

\begin{theorem} \label{ellipth2} Let $P_{j}$ be elliptic self-adjoint strictly positive pseudo-differential operator as above, $j=1,\dots,p$,. Suppose that there is $l\in\{1,2,\dots,p\}$ such that $n_{l}/m_{l}>n_{j}/m_{j}$ for all $j\neq l$. Then, the spectral counting function $N_{P}$ of the operator $P=P_{1}\otimes P_{2}\otimes \dots \otimes P_{p}$ has asymptotics
\begin{equation}
\label{ellipeq5}
N_{P}(\lambda)=\sum_{\lambda^{(1)}_{k_1}\lambda^{(2)}_{k_2}\dots\lambda^{(p)}_{k_p}\leq\lambda}1=\left(A_{l}\prod_{j\neq l}\zeta_{P_{j}}(n_{l}/m_{l})\right) \lambda^{n_{l}/m_{l}}+ O(\lambda^{\tau}), \ \ \ \lambda\to\infty,
\end{equation}
where $A_{l}$ is given by $(\ref{ellipeq4})$ and $\tau$ satisfies $\max\{(n_{l}-1)/m_{l},\max_{j\neq l}n_{j}/m_{j}\}<\tau<n_{l}/m_{l}$.
\end{theorem}

For the special case of the tensor product of two elliptic operators with one counting function dominating the other one, the error term in (\ref{eq32}) improves that from \cite[Thrm. 3.2 ]{batisti} for bisingular operators.

We leave to the reader statements and proofs for the counterparts of Theorem \ref{ellipth1}, (\ref{eq315}) and Theorem \ref{th37} in the setting of closed manifolds.

\end{document}